\documentclass[10pt]{amsart}
\usepackage[cp1251]{inputenc}
\usepackage[english,russian]{babel}
\usepackage{amsmath}
\usepackage{amssymb}
\usepackage{amsfonts}
\usepackage{graphicx}
\usepackage{tikz}
\usepackage[hyphens]{url}
\usepackage{hyperref}
\usepackage{graphicx}
\usepackage{tikz} % пїЅ Р°Р±РѕС‚Р° СЃ РіСЂР°С„РёРєРѕР№
\usepackage{textcomp}
\usepackage{circuitikz}
\usetikzlibrary{math} 
\usepackage{caption}
\usepackage{subcaption}
\usepackage{pgfplots}
\usepackage{pgfplotstable}
\usepackage{amssymb}

\DeclareMathOperator{\Int}{Int}

\DeclareMathOperator{\card}{card}

\newtheorem{lemma}{Lemma}
\newtheorem{theorem}{Theorem}
\newtheorem{corollary}{Corollary}
\newtheorem{proposition}{Proposition}

\theoremstyle{remark}
\newtheorem{definition}{Definition}
\newtheorem{remark}{Remark}

\DeclareMathOperator{\wind}{wind}

\begin{document}
	\renewcommand{\refname}{References}
	\renewcommand{\proofname}{Proof.}
	\renewcommand{\figurename}{Fig.}

	\thispagestyle{empty}

	\title[HOPF-TYPE THEOREMS FOR $f$-NEIGHBORS]{HOPF-TYPE THEOREMS FOR $f$-NEIGHBORS}
	\author{{A.V. Malyutin}, {I.M. Shirokov}}%
	\address{Andrei Valeryevich Malyutin
		\newline\hphantom{iii} St. Petersburg Department
		\newline\hphantom{iii} of Steklov Mathematical Institute,
		%of Russian Academy of Sciences,
		\newline\hphantom{iii} Fontanka, 27,
		\newline\hphantom{iii} 191023, St. Petersburg, Russia
		\newline\hphantom{iii} St. Petersburg State University,
        \newline\hphantom{iii} Universitetskaya Emb., 13B,
        \newline\hphantom{iii} 199034, St. Petersburg, Russia}%
        \email{malyutin@pdmi.ras.ru}%
	\address{Ilya Mikhailovich Shirokov
		\newline\hphantom{iii} St. Petersburg Department
		\newline\hphantom{iii} of Steklov Mathematical Institute,
		%of Russian Academy of Sciences,
		\newline\hphantom{iii} Fontanka, 27,
		\newline\hphantom{iii} 191023, St. Petersburg, Russia}%
	\email{shirokov.im@phystech.edu}%
	
	\thanks{\sc Malyutin, A.V., Shirokov, I.M.,
		Hopf-type theorems for $f$-neighbors}
	\thanks{\copyright \ 2022 Malyutin A.V., Shirokov I.M}
	\thanks{\rm This research project was started during the Summer Research Program for Undergraduates 2021 organized by the Laboratory of Combinatorial and Geometric Structures at MIPT. %?? This program was funded by the Russian Federation Government. 
%??(Grant number 075-15-2019-1926).
		This research was supported by the Ministry of Science and Higher Education of the Russian Federation, agreement 075-15-2019-1620 date 08/11/2019 and 075-15-2022-289 date 06/04/2022.	
		%?? The first author was partially supported by RFBR according to the research project n. 20-01-00070.
		}
	%\thanks{\it Received  January, 1, 2015, published  March, 1,  2015.}%

%	\semrtop \vspace{1cm}
	\maketitle {\small
		\begin{quote}
			\noindent{\sc Abstract. } 
We work within the framework of a program aimed at exploring various extended versions for theorems from a class containing Borsuk--Ulam type theorems, some fixed point theorems, the KKM lemma, Radon, Tverberg, and Helly theorems.
%We explore various extended versions for theorems from a class containing Borsuk--Ulam type theorems, some fixed point theorems, the KKM lemma, etc. 
In this paper we study variations of the Hopf theorem concerning continuous maps of a compact Riemannian manifold $M$ of dimension $n$ to $\mathbb{R}^n$. We investigate the case of maps $f\colon M \to \mathbb{R}^m$ with $n < m$ and introduce several notions of varied types of $f$-neighbors, which is a pair of distinct points in $M$ such that $f$ takes it to a ‘small’ set of some type. Next for each type, we ask what distances on $M$ are realized as distances between $f$-neighbors of this type and study various characteristics of this set of distances. One of our main results is as follows. Let $f\colon M \to \mathbb{R}^{m}$ be a continuous map. We say that two distinct points $a$ and $b$ in $M$ are visual $f$-neighbors if the segment in $\mathbb{R}^{m}$ with endpoints $f(a)$ and $f(b)$ intersects $f(M)$ only at $f(a)$ and $f(b)$. Then the set of distances that are realized as distances between visual $f$-neighbors is infinite. Besides we generalize the Hopf theorem in a quantitative sense. 

			\medskip
			
		\end{quote}
	}

	\section{Introduction}
	
	In this paper we continue to work on the program, started in~\cite{MM18}, of exploring various extended versions for the so-called ‘point-type’ theorems. By the ‘point-type’ theorems we mean those in which a singleton or a point plays a central role. This class of theorems includes in particular the following subclasses.
	
	\begin{enumerate}
		\item Fixed-point theorems (Brouwer, Kakutani, etc.);
		\item Theorems on the existence of common points for subsets
		(Radon, Tverberg, Helly, Knaster--Kuratowski--Mazurkiewicz, etc.);
		\item Theorems on taking sets of a large diameter to a singleton (Borsuk--Ulam, Hopf, topological versions of Radon and Tverberg theorems, etc.).
	\end{enumerate}
	
	The first driving idea of the program is to cover the cases of arbitrary dimensions and new classes of spaces by relaxing the ‘point-type’ conditions. For example, we replace conditions of the forms ‘the map has a fixed point,’ ‘subsets have a common point,’ and ‘the image of a subset that is large in some sense is a singleton’ with milder ones ‘the map has a point that moves slightly in some sense,’ ‘subsets have an equidistant point,’ and ‘the image of a subset that is large in some sense is a subset that is small in some sense’ respectively.
	
	The Borsuk--Ulam and Radon theorems, as well as the Knaster--Kuratowski--Mazurkiewicz lemma, have counterparts obtained in this way (see~\cite{MM18}). In this paper we operate in the Borsuk--Ulam subclass of theorems (3) and concentrate on extensions related to the Hopf theorem \cite{H44}:
	
	\begin{theorem}[H. Hopf \cite{H44}]
	\label{th:Hopf}
		Let $n$ be a positive integer, let $M$ be a compact Riemannian manifold of dimension $n$, and let $f\colon M \to \mathbb{R}^n$ be a continuous map. Then for any prescribed $\delta > 0$, there exists a pair $\{x,y\} \in M \times M$ such that $f(x) = f(y)$ and $x$ and $y$ are joined by a geodesic of length~$\delta$.
	\end{theorem}
	
	The Hopf theorem has a rigid restriction on the codomain dimension and involves the strict condition $f(x)=f(y)$. A~question we are interested in is whether the maps not satisfying dimension restrictions give sufficiently large collections of pairs $\{x,y\} \in M \times M$ if we replace the requirement $f(x) = f(y)$ with a requirement of the form ‘$f(x)$ and $f(y)$ are in some sense close to each other.’ We examine several versions of ‘proximity to each other,’ which are presented in the following definition of $f$-neighbors of various types.
%To formalize the idea of `proximity to each other' we introduce several notions of $f$-neighbors of various types.
	
	\begin{definition}
	\label{def:f-neighbors}
		Let $(\mathcal{M}, d)$ be a metric space with distance function~$d$, and let $f\colon \mathcal{M} \to \mathbb{R}^m$ be a continuous map. 
		Distinct points $a$ and $b$ in $\mathcal{M}$ are called
		\begin{enumerate}
			\item \emph{$f$-neighbors} if $f(a) = f(b)$;
			\item \emph{spherical $f$-neighbors} if either $f(a) = f(b)$ or there exists a Euclidean ball $B^m \subset \mathbb{R}^m$ such that\footnote{We use $\partial X$ and $\Int X$ to denote the boundary and interior, respectively, of a subset~$X$ in a topological space.} $\{f(a), f(b)\} \subset \partial B^m $ and $f(\mathcal{M}) \cap \Int B^m = \emptyset$;
			\item \emph{visual $f$-neighbors} if either $f(a) = f(b)$ or the segment with endpoints $f(a)$ and $f(b)$ intersects $f(\mathcal{M})$ only at these two points;
			\item \emph{topological $f$-neighbors} if either $f(a) = f(b)$ or some path in $\mathbb{R}^{m}$ with endpoints at $f(a)$ and $f(b)$ intersects $f(\mathcal{M})$ only at these two points.
		\end{enumerate}
	\end{definition}
	
    To ‘measure’ collections of $f$-neighbors of various types we use the sets of distances that are realized as distances between $f$-neighbors.  
	
	\begin{definition}
		We introduce the following notation:
		\begin{enumerate}
			\item $\Omega_f = \{d(a,b) \in \mathbb{R} \:|$ $a$ and $b$ in $\mathcal{M}$ are $f$-neighbors$\}$; 
			\item $\Omega_f^{sph} = \{d(a,b) \in \mathbb{R} \:|$ $a$ and $b$ in $\mathcal{M}$ are spherical $f$-neighbors$\}$;
			\item $\Omega_f^{vis} = \{d(a,b) \in \mathbb{R} \:|$ $a$ and $b$ in $\mathcal{M}$ are visual $f$-neighbors$\}$; 
			\item $\Omega_f^{top} = \{d(a,b) \in \mathbb{R} \:|$ $a$ and $b$ in $\mathcal{M}$ are topological $f$-neighbors$\}$.
		\end{enumerate}
	\end{definition}
	
	Then $\Omega_f \subseteq \Omega_f^{sph} \subseteq \Omega_f^{vis} \subseteq \Omega_f^{top}$. 
	
	The Hopf theorem implies that, in its settings and with respect to the natural path-metric on~$M$, each of the sets $\Omega_f$, $\Omega_f^{sph}$, $\Omega_f^{vis}$, and $\Omega_f^{top}$ is continual and contains an interval adjacent to~$0$. Clearly, any embedding $f\colon M \to \mathbb{R}^m$ yields empty~$\Omega_f$. The following theorem (see \cite{MM18}) shows that $\Omega_f^{sph}$ is nonempty for the case of continuous maps $f\colon\mathbb{S}^n \to \mathbb{R}^m$ with $n < m$.
	
	\begin{theorem}[A.\,V. Malyutin, O.\,R. Musin \cite{MM18}]
		Let $m$ and $n$ be positive integers such that $n < m$, let~\,$\mathbb{S}^n$ be a Euclidean unit $n$-sphere in the Euclidean $(n+1)$-space $\mathbb{R}^{n+1}$, and let~\,$\mathbb{S}^n \to \mathbb{R}^m$ be a continuous map. Then there are spherical $f$-neighbors $x$ and $y$ in $\mathbb{S}^n$ such that the Euclidean distance $|x-y|$ is at least $\sqrt{(n+2)/n}$.
	\end{theorem}
	
	The main point of the present paper is further study of properties of sets $\Omega_f^{sph}, \Omega_f^{vis}$, and $\Omega_f^{top}$ for the case of continuous maps $f\colon M\to \mathbb{R}^m$, where $M$ is a compact Riemannian manifold of dimension $\dim M <m$. 
	
	The paper is organised as follows. 
	First in section ‘Negative results’ we give some examples showing that in general even $\Omega_f^{top}$ can contain large ‘holes’ and be arbitrarily ‘small’ in some sense. 
	Next, in section ‘Positive results,’ we prove that $\Omega_f^{vis}$ is infinite for continuous maps $f\colon M \to \mathbb{R}^m$ of compact Riemannian manifolds (with any intrinsic metrics compatible with the topology of~$M$), and obtain the same result for $\Omega_f^{sph}$ in the case of continuous maps $f\colon \mathbb{S}^1 \to \mathbb{R}^2$.	
	In the last section we generalize the Hopf theorem in a quantitative sense.
	
	\section{Negative results}
	
	%In this section we will denote by $d$ the angular distance, induced by the Euclidean metric in $\mathbb{R}^{n+1}$, between points of $\mathbb{S}^{n}$. 
	First we give an example of a continuous map $f\colon \mathbb{S}^1 \to \mathbb{R}^2$ where, unlike in the Hopf theorem, the set $\Omega_f^{vis}$ does not cover the interval $(0,\pi]$. 
	
	\begin{proposition}[Example 1]
	Let~\,$\mathbb{S}^1$ be a Euclidean circle in~$\mathbb{R}^2$, and let $d$ be the angular distance on~\,$\mathbb{S}^1$. Then for any prescribed $\varepsilon \in (0, \frac{\pi}{3})$, there exists a continuous map $f_{\varepsilon}\colon \mathbb{S}^1 \to \mathbb{R}^{2}$ such that $\Omega_{f_\varepsilon}^{vis} = (0, \varepsilon] \cup [\frac{2\pi}{3} - \varepsilon, \pi]$.
	\end{proposition}
	
	\begin{figure}[h]
	\label{fig:trefoil}
		\centering
		\begin{tikzpicture}
			
			\begin{scope}[scale = 1]
				
				\node[scale=2] at (0,2.2) {\color{blue} \tiny$\alpha$};
				\node[scale=2] at ({2*0.866+0.15},-1.15) {\color{blue} \tiny$\alpha$};
				\node[scale=2] at ({-2*0.866-0.15},-1.15) {\color{blue} \tiny$\alpha$};
				
				\node[scale=2] at ({2*0.866-0.15},1 - 0.15) {\color{red} \tiny$\varepsilon$};
				\node[scale=2] at ({-(2*0.866-0.15)},1 - 0.15) {\color{red} \tiny$\varepsilon$};
				\node[scale=2] at (0,-1.8) {\color{red} \tiny$\varepsilon$};
				
				\draw [red,thick,domain=20:40] plot ({2*cos(\x)}, {2*sin(\x)});
				\draw [red,thick,domain=140:160] plot ({2*cos(\x)}, {2*sin(\x)});
				\draw [red,thick,domain=260:280] plot ({2*cos(\x)}, {2*sin(\x)});
				
				\draw [blue,thick,domain=-80:20] plot ({2*cos(\x)}, {2*sin(\x)});
				\draw [blue,thick,domain=40:140] plot ({2*cos(\x)}, {2*sin(\x)});
				\draw [blue,thick,domain=160:260] plot ({2*cos(\x)}, {2*sin(\x)});
				
				\fill[black] ({2*cos(20}, {2*sin(20)})  circle (1pt) node [right] {\footnotesize $A$};
				\fill[black] ({2*cos(40}, {2*sin(40)})  circle (1pt) node [right] {\footnotesize $B$};
				\fill[black] ({2*cos(140}, {2*sin(140)})  circle (1pt) node [left] {\footnotesize $C$};
				\fill[black] ({2*cos(160}, {2*sin(160)})  circle (1pt) node [left] {\footnotesize $D$};
				\fill[black] ({2*cos(260}, {2*sin(260)})  circle (1pt) node [below] {\footnotesize $F$};
				\fill[black] ({2*cos(280}, {2*sin(280)})  circle (1pt) node [below] {\footnotesize $G$};
			\end{scope}
			
			\begin{scope}[scale = 0.5, xshift = 8cm]
				
				\draw[thick, blue, scale=4] (1,1) to [out=180,in=90] (0,0);
				\draw[thick, blue, scale=4] (0,0) to [out=-90,in=180] (1,-1) to [out = 0, in = 270] (2,0);
				\draw[thick, blue, scale=4] (2,0) to [out = 90, in = 0] (1,1);
				
				\draw[thick, red, scale=4] (1.86605, 0.5) to [out=125, in=55] (0.1339,0.5);
				\draw[thick, red, scale=4] (0.1339,0.5) to [out=245, in=175] (1,-1);
				\draw[thick, red, scale=4] (1,-1) to [out=5, in=295] (1.86605,0.5);
				
				\fill[black] (4*1.86605, 4*0.5)  circle (2pt) node [right] {\footnotesize $A, D$};
				\fill[black] (4*0.1339, 4*0.5)  circle (2pt) node [left] {\footnotesize $B, F$};
				\fill[black] (4*1, -4)  circle (2pt) node [below] {\footnotesize $C, G$};
				
				\node[scale=2] at (4,3.2) {\color{red} \tiny$\varepsilon$};
				\node[scale=2] at (1,-1) {\color{red} \tiny $\varepsilon$};
				\node[scale=2] at (7,-1) {\color{red} \tiny $\varepsilon$};
				
				\node[scale=2] at (4,4.4) {\color{blue} \tiny $\alpha$};
				\node[scale=2] at (0,-1.8) {\color{blue} \tiny $\alpha$};
				\node[scale=2] at (8,-1.8) {\color{blue} \tiny $\alpha$};
				
			\end{scope}
		\end{tikzpicture}
		\caption{\:$\mathbb{S}^1$ left and $f_{\varepsilon}(\mathbb{S}^1)$ right}
		
	\end{figure}
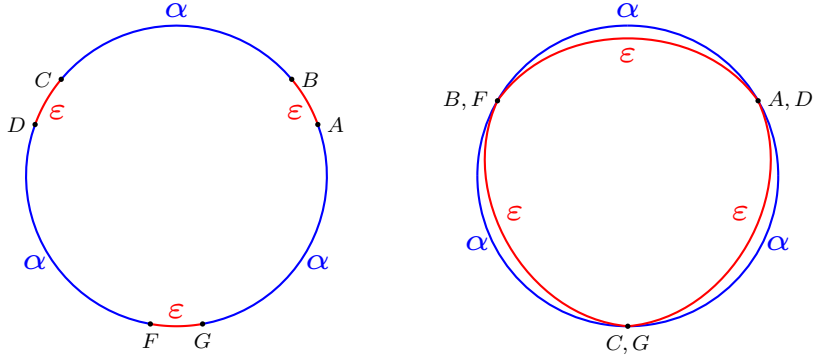
	
	\begin{proof}
A construction of such $f_{\varepsilon}$ is shown in Fig.~1. The parametrization of~$f_{\varepsilon}$ on each of the red and blue arcs is uniform. To find $\Omega_{f_\varepsilon}^{vis}$ we study which boundary points of each connected component of $\mathbb{R}^2 \setminus f(\mathbb{S}^1)$ can ‘see’ each other. It is enough to study a connected component whose boundary is the union of a red $\varepsilon$-arc and a blue $\alpha$-arc, and the component whose boundary is the union of red $\varepsilon$-arcs. If red $\varepsilon$-arcs are sufficently close to the corresponding blue $\alpha$-arcs, then we get $$\Omega_{f_\varepsilon}^{vis} = (0, \varepsilon] \cup \left[\frac{2\pi}{3} - \varepsilon, \pi\right].\qedhere$$ 
	\end{proof}
	
	Next we are going to show that there exist continuous maps $f\colon \mathbb{S}^n \to \mathbb{R}^{n+1}$ with $\Omega_f^{top}$ of arbitrarily small Lebesgue measure. 
	
	\begin{proposition}[Example 2]
	\label{prop:Lebesgue}
		Let $n$ be a positive integer, let $\mathbb{S}^n$ be a Euclidean $n$-sphere in the Euclidean $(n+1)$-space $\mathbb{R}^{n+1}$, and let $d$ be the angular distance on~$\mathbb{S}^n$. Then there exists a sequence of continuous maps $(f_{k}\colon \mathbb{S}^n \to \mathbb{R}^{n+1})_{k=1}^{\infty}$ such that $\lim\limits_{k \to \infty} \mu(\Omega_{f_{k}}^{top}) = 0$, where $\mu$ is the Lebesgue measure on $\mathbb{R}$.
	\end{proposition}
	
	\begin{proof}			
		Let $k\ge 2$ be a positive integer. Then there exists a collection $P_k = \{S_l\}_{l=1}^{k}$ of subsets of~$\mathbb{S}^n$ such that each $S_l$ is a closed topological $n$-ball, $\mathbb{S}^n = \bigcup_{l=1}^{k} S_l$, and $S_l \cap S_m = \partial{S_l} \cap \partial{S_m}$ if $l \neq m$. (Say, $P_2$ is a presentation of~$\mathbb{S}^n$ as the union of two hemispheres, and for $k\ge 3$ one can obtain $P_k$ by replacing an $n$-ball in $P_{k-1}$ with two topological $n$-balls.) 
		For each $l \in \{1,\dots,k\}$, choose a point $\xi_l \in$ $\Int S_l$.
		Then choose $\varepsilon \in (0, 1/k^3)$ such that for each $l \in \{1,\dots,k\}$, $S_l$ contains the $\varepsilon$-neighborhood $U_{\varepsilon}(\xi_l)$ of $\xi_l$ in $\mathbb{S}^n$. 
		
		We construct our $f_k$ as follows. Set $f_k$ to be the identity map on $\mathbb{S}^n \setminus U$, where $U = \bigcup_{l=1}^{k}U_{\varepsilon}(\xi_l)$. Let $T_l$ be a ‘two-side’ thickening of $S_l$ along the radius of $\mathbb{S}^n$. Define $f_k$ on $U_{\varepsilon}(\xi_l)$ such that $f_k(U_{\varepsilon}(\xi_l)) = T_l$.
		
		Then $\mathbb{S}^n$ is ‘surrounded’ by the sets $f_k(U_{\varepsilon}(\xi_l))$ and $f_k(U_{\varepsilon}(\xi_l))$ ‘hides’ $S_l\setminus U_{\varepsilon}(\xi_l)$ from other $S_m\setminus U_{\varepsilon}(\xi_m)$. 
		
		Clearly, $\Omega_{f_k}^{top}$ is contained in the union of the interval $(0, D_k^{max})$, where $D_k^{max}$ is the maximum of diameters of~$S_l$, $l \in \{1, \dots, k\}$, and the intervals of the form $$(d(\xi_i, \xi_j) - 2/k^3, d(\xi_i, \xi_j) + 2/k^3), \quad i, j \in \{1, \dots, k\}.$$ Thus we have 
		$$\mu(\Omega_{f_k}^{top}) \leq D_k^{max} + 4k(k-1)/k^3.$$ 
		
		Taking $P_k$ with increasing $k$ and with $D_k^{max}$ tending to $0$ as $k$ tends to infinity, we obtain a sequence of maps with the desired property.
	\end{proof}

The construction of the above proof of Proposition~\ref{prop:Lebesgue} implies the following proposition.	
	
	\begin{proposition}
Under assumptions of Proposition~\ref{prop:Lebesgue}, for any finite set $F\subset \mathbb{R}$ there exists a continuous map $f_{F} \colon \mathbb{S}^n \to \mathbb{R}^{n+1}$ such that $F \cap \Omega_{f_F}^{top}=\varnothing$.
	\end{proposition}
	
	%It is worth to note that there is an 
	We describe an example %? construction ?
	of another kind with the same result as above for the case of continuous maps $f\colon \mathbb{S}^1 \to \mathbb{R}^2$. Namely, for each pair of coprime integers $(3,q)$ with $q \geq 4$, we define a continuous map $f_q \colon \mathbb{S}^1 \to T_{3,q}$, where $T_{3,q}\subset\mathbb{R}^2$ is a standard diagram of torus knot of type $(3,q)$ (see Fig.~2).
	
	\begin{figure}[h]
		\centering
		\begin{tikzpicture}[>=latex]
			\begin{scope}[scale = 0.7]
				
				\draw [red,thick,domain=-10:10] plot ({2*cos(\x)}, {2*sin(\x)});
				\draw [blue,thick,domain=10:80] plot ({2*cos(\x)}, {2*sin(\x)});
				\draw [red,thick,domain=80:100] plot ({2*cos(\x)}, {2*sin(\x)});
				\draw [blue,thick,domain=100:170] plot ({2*cos(\x)}, {2*sin(\x)});
				\draw [red,thick,domain=170:190] plot ({2*cos(\x)}, {2*sin(\x)});
				\draw [blue,thick,domain=190:260] plot ({2*cos(\x)}, {2*sin(\x)});
				\draw [red,thick,domain=260:280] plot ({2*cos(\x)}, {2*sin(\x)});
				\draw [blue,thick,domain=280:350] plot ({2*cos(\x)}, {2*sin(\x)});
				
				\fill[black] ({2*cos(10)}, {2*sin(10)})  circle (1.5pt) node [right] {\footnotesize $B$};
				\fill[black] ({2*cos(10+90)}, {2*sin(10+90)})  circle (1.5pt) node [above] {\footnotesize $D$};
				\fill[black] ({2*cos(10+90+90)}, {2*sin(10+90+90)})  circle (1.5pt) node [left] {\footnotesize $F$};
				\fill[black] ({2*cos(10+270)}, {2*sin(10+270)})  circle (1.5pt) node [below] {\footnotesize $H$};
				
				\fill[black] ({2*cos(-10)}, {2*sin(-10)})  circle (1.5pt) node [right] {\footnotesize $A$};
				\fill[black] ({2*cos(-10+90)}, {2*sin(-10+90)})  circle (1.5pt) node [above] {\footnotesize $C$};
				\fill[black] ({2*cos(-10+90+90)}, {2*sin(-10+90+90)})  circle (1.5pt) node [left] {\footnotesize $E$};
				\fill[black] ({2*cos(-10+270)}, {2*sin(-10+270)})  circle (1.5pt) node [below] {\footnotesize $G$};
				
			\end{scope}
			
			\begin{scope}[scale = 0.65, xshift=7cm]
				\node[right] at (0.5, 0)  {\footnotesize $A, D$};
				\node[below] at (0, -0.7)  {\footnotesize $C, F$};
				\node[above] at (0, 0.7)  {\footnotesize $B, G$};
				\node[left] at (-0.55, 0)  {\footnotesize $H, E$};

				\draw[thick,blue] (2,2) -- (0,0.5) node[currarrow, pos=0.5, xscale=-1.7, sloped, scale=0.7, color = blue] {};
				\draw[thick,red] (0,0.5) -- (-0.5,0) node[currarrow, pos=0.5, xscale=-1.7, sloped, scale=0.7,color = red] {};
				\draw[thick,blue] (-0.5,0) -- (-2,-2) node[currarrow, pos=0.5, xscale=-1.7, sloped, scale=0.7, color = blue] {};
				\draw[thick,blue] (-2,-2) -- (2,-2) node[currarrow, pos=0.5, xscale=1.7, sloped, scale=0.7, color = blue] {};
				\draw[thick,blue] (2,-2) -- (0.5,0) node[currarrow, pos=0.5, xscale=-1.7, sloped, scale=0.7, color = blue] {};
				\draw[thick,red] (0.5,0) -- (0,0.5) node[currarrow, pos=0.5, xscale=-1.7, sloped, scale=0.7, color = red] {};
				\draw[thick,blue] (0,0.5) -- (-2,2) node[currarrow, pos=0.5, xscale=-1.7, sloped, scale=0.7, color = blue] {};
				\draw[thick,blue] (-2,2) -- (-2,-2) node[currarrow, pos=0.5, xscale=1.7, sloped, scale=0.7,color = blue] {};
				\draw[thick,blue] (-2,-2) -- (0,-0.5) node[currarrow, pos=0.5, xscale=1.7, sloped, scale=0.7,color = blue] {};
				\draw[thick,red] (0,-0.5) -- (0.5,0) node[currarrow, pos=0.5, xscale=1.7, sloped, scale=0.7, color = red] {};
				\draw[thick,blue] (0.5,0) -- (2,2) node[currarrow, pos=0.5, xscale=1.7, sloped, scale=0.7, color = blue] {};
				\draw[thick,blue] (2,2) -- (-2,2) node[currarrow, pos=0.5, xscale=-1.7, sloped, scale=0.7, color = blue] {};

				\draw[thick,blue] (-2,2) -- (-0.5,0) node[currarrow, pos=0.5, xscale=1.7, sloped, scale=0.7, color = blue] {};
				\draw[thick,red] (-0.5,0) -- (0,-0.5) node[currarrow, pos=0.5, xscale=1.7, sloped, scale=0.7, color = red] {};
				\draw[thick,blue] (0,-0.5) -- (2,-2) node[currarrow, pos=0.5, xscale=1.7, sloped, scale=0.7, color = blue] {};
				\draw[thick,blue]  (2,-2) -- (2,2) node[currarrow, pos=0.5, xscale=1.7, sloped, scale=0.7, color = blue] {};
				
				\fill[black] (0.5, 0)  circle (1.5pt); 
				\fill[black] (0, 0.5)  circle (1.5pt); 
				\fill[black] (-0.5, 0)  circle (1.5pt); 
				\fill[black] (0, -0.5)  circle (1.5pt);
			\end{scope}
		\end{tikzpicture}
		\caption{Example of $f_{4}$: $\mathbb{S}^1$ left and $f_{\alpha}(\mathbb{S}^1)$ right}
	\end{figure}
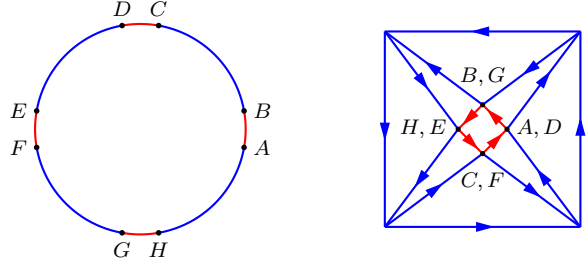
	
	We take red arcs with centers at vertices of a regular $q$-polygon inscribed in~$\mathbb{S}^1$ and construct our $f_q$ with uniform parametrization on each of red and blue arcs as shown in Fig.~2. The image of the union of red arcs forms a regular $q$-polygon, and the endpoints of the image of each blue arc are vertices of this $q$-polygon. We construct a sequence $(f_q)_{q=4}^{\infty}$, with $q$ coprime to~$3$, such that the total length of red arcs decreases as $O(1/q^3)$ as $q$ tends to infinity. 
	
	One can find $\Omega_{f_q}^{sph}$, $\Omega_{f_q}^{vis}$, and $\Omega_{f_q}^{top}$ by studying ‘colorings’ of the set of pairs of points $T^2 = \mathbb{S}^1 \times \mathbb{S}^1$. Namely, we color a pair $(a,b)$ red if $a$ and $b$ are $f$-neighbors of corresponding type and green otherwise. Fig.~3 with an unfolding of colored torus provides an example for $\Omega_{f_4}^{vis}$, where we color a pair $(a,b)$ red if $f(a)$ and $f(b)$ ‘see’ each other.
	
	\begin{figure}[h]
		\begin{center}
			\includegraphics[width=0.25\linewidth]{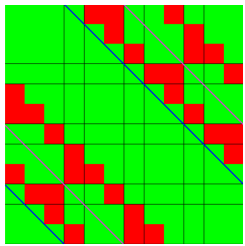}\\
			\caption{Coloring of $\mathbb{S}^1 \times \mathbb{S}^1$}
		\end{center}
	\end{figure}
	
	In Fig.~3, for all points on a line that is parallel to a diagonal, angular distances between corresponding points on $\mathbb{S}^1$ are the same. When we increase $q$, the number of red squares increases and we get more ‘red stripes,’ but they become narrower and the total area of red squares decreases. As a result we have ${\mu(\Omega_{f_{q}}^{vis}) \to 0}$. 
	
	A variation of this construction and similar arguments work for $\Omega_{f_{q}}^{top}$.
	
	It would be interesting to apply the ‘colored torus’ method to study the sets $\Omega_{f}^{sph}$, $\Omega_{f}^{vis}$, and $\Omega_{f}^{top}$ for the case of more complex closed braid diagrams.

	\section{Positive results}
	
	In the previous section we saw that even $\Omega_{f}^{top}$ can have arbitrarily small Lebesgue measure. It is natural to ask, Can $\Omega_{f}^{top}$, $\Omega_{f}^{vis}$, or $\Omega_{f}^{sph}$ be finite? In this section we first prove that $\Omega_{f}^{vis}$ is infinite for the case of continuous maps $f\colon M \to \mathbb{R}^m$, where $M$ is a compact Riemannian manifold (with an intrinsic metric compatible with the topology on $M$), and then we show that $\Omega_{f}^{sph}$ is infinite for the case of continuous maps $f\colon\mathbb{S}^1 \to \mathbb{R}^2$. We need several auxiliary lemmas and definitions.
	
	%%%% any smooth manifold that is paracompact and Hausdorff admits a Rimannian metric. Should we change formulations to smooth paracompact manifolds?
	
		\begin{definition}
	\label{def:locally-injective}
		Let $A$ and $B$ be topological spaces. A continuous map $f\colon A\to B$ is called \emph{locally injective} if any point in $A$ has a neighborhood such that the restriction of $f$ to this neighborhood is injective.
	\end{definition}
	
		\begin{definition}[cf.~Definition~\ref{def:f-neighbors}]
	\label{def:visual neighbors}
Let $Y$ be a subset of $\mathbb{R}^m$. We say that two points $p$ and $q$ in~$Y$ are %\emph{visual neighbors} if $p\neq q$ and the segment with endpoints $p$ and $q$ intersects $Y$ only at these two points.
	\begin{enumerate}
			\item \emph{visual neighbors} (with respect to~$Y$) if $p\neq q$ and the segment with endpoints $p$ and $q$ intersects $Y$ only at these two points;
			\item \emph{spherical neighbors} (with respect to~$Y$) if $p\neq q$ and there exists a Euclidean ball $B^m \subset \mathbb{R}^m$ such that $\{p, q\} \subset \partial B^m$ and $Y \cap \Int B^m = \emptyset$.
		\end{enumerate}
	\end{definition}

		\begin{proposition}
	\label{prop:locally-injective-properties}
		Let $f\colon X \to Z$ be a continuous map of metric spaces.
		If $X$ is compact and $f$ is locally injective, then
		\begin{enumerate}
			\item For any $z\in f(X)$, the inverse image $f^{-1}(z)$ is finite;  
			%\item There exists $\delta>0$ such that no $f$-neighbors $a$ and $b$ have $d(a,b) < \delta$ (where $d$ stands for the distance function on~$X$);
			\item Each point $z\in f(X)$ has a neighborhood $U=U(z)$ such that the restriction of~$f$ to each connected component of $f^{-1}(U)$ is injective.
		\end{enumerate}
	\end{proposition}
	
		\begin{proof}
		\begin{enumerate}
			\item Assume that $f^{-1}(z)$ is infinite for some $z\in f(X)$. Then $f^{-1}(z)$ has an accumulation point since $X$ is compact. Clearly, no accumulation point of $f^{-1}(z)$ has a neighborhood such that the restriction of $f$ to this neighborhood is injective.
			
			%\item Assume to the contrary that for each positive integer~$i$ there exist $f$\nobreakdash-neighbors $a_i$ and $b_i$ with $d(a_i,b_i) < 1/i$. Since $X$ is compact, the sequence $(a_i)_{i=1}^{\infty}$ has converging subsequences. Clearly, no limit point of a converging subsequence of $(a_i)_{i=1}^{\infty}$ has a neighborhood such that the restriction of $f$ to this neighborhood is injective.
			
			\item By assertion~(1) of the proposition, $f^{-1}(z)$ is finite. Suppose $f^{-1}(z)$ consists of $k$ points and denote them by $x_1, \dots, x_k$. Due to the assumed local injectivity of $f$, each of $x_i$'s has an open neighborhood $U_i$ such that the restriction of $f$ to $U_{i}$ is injective. Taking (if necessary) smaller neighborhoods of $x_1$, $\dots$, $x_k$, we can assume that $U_1$, $\dots$, $U_k$ are pairwise disjoint. Observe that the complement $C:=X\setminus(U_1\cup \dots \cup U_k)$ is compact and $z\notin f(C)$. Hence, since a continuous image of a compact set is compact, it follows that there is a neighborhood $W = W(z)$ of~$z$ such that $f(C)\cap W(z) = \emptyset$. Therefore, $f^{-1}(W)$ is contained in $U_1\cup \dots \cup U_k$. Since $U_1$, $\dots$, $U_k$ are all open and pairwise disjoint, it follows that each connected component of $f^{-1}(W)$ is contained in $U_i$ for some~$i$. The statement follows. \qedhere
		\end{enumerate} 
	\end{proof}

\begin{proposition}
	\label{prop:non-locally-injective}
		Let $f\colon X \to Z$ be a continuous map of metric spaces. 
If $X$ is compact then the following conditions are equivalent:
		\begin{enumerate} 
		\item $f$ is not locally injective; 
		\item There exist arbitrarily close $f$-neighbors.
		\end{enumerate} 
\end{proposition}
	
	\begin{proof}
If $f$ is not locally injective then there exists $x\in X$ such that for each positive integer~$i$ the open metric ball $B_{1/i}(x)$ of radius $1/i$ centered at~$x$ contains distinct points $a_i$ and $b_i$ with $f(a_i)=f(b_i)$. These $a_i$ and $b_i$ are $f$-neighbors within distance at most $2/i$ from each other.

Conversely, suppose there exist arbitrarily close $f$-neighbors. This means that for each positive integer~$i$, there exist $f$\nobreakdash-neighbors $a_i$ and $b_i$ with $d(a_i,b_i) < 1/i$. Since $X$ is compact, the sequence $(a_i)_{i=1}^{\infty}$ has converging subsequences. Clearly, no limit point of a converging subsequence of $(a_i)_{i=1}^{\infty}$ has a neighborhood such that the restriction of $f$ to this neighborhood is injective.	
	\end{proof}

	\begin{lemma}
	\label{lem:properly-covered}
		Let $X$ be a topological space, let $k$ be a positive integer, and let $(A_1, \dots, A_k)$ be a $k$-element sequence of closed sets in~$X$. Suppose that at least one of these sets is nonempty. A point $P$ in the union $A:=A_1\cup \dots \cup A_k$ is called \emph{properly covered} if there exists a neighborhood $U = U(P)$ of $P$ such that all nonempty sets in the sequence $(U\cap A_1, \dots, U\cap A_k)$ coincide. Then the set of properly covered points is open and everywhere dense in $A$.
	\end{lemma}
	
	\begin{proof}
		First we prove that $A$ contains at least one properly covered point. We use induction on $k$. For $k = 1$ the statement is obvious. Let $k > 1$. If all nonempty sets in $(A_1,\dots, A_k)$ coincide, then taking the whole space $X$ as a suitable neighborhood, we get that all points in $A$ are properly covered.
		%
		%% "the" whole space? -- google scholar yes
		%
		Otherwise there are nonempty noncoinciding sets in $(A_1, \dots, A_k)$, and hence there exists $j \in \{1,\dots,k\}$ such that $A_{j}$ is nonempty and $A_{j} \neq A$.	
		Then at least one element in the $(k-1)$-element sequence $$(A_1\setminus A_j, \dots, A_{j-1}\setminus A_j, A_{j+1}\setminus A_j, \dots, A_k\setminus A_j)$$ is nonempty and all elements of this sequence are closed subsets in $X\setminus A_j$.

		Applying the inductive hypothesis to the space $X\setminus A_j$ and the $(k-1)$-element sequence above, we get that there exists a properly covered (in the corresponding sense) point $P$ in $A\setminus A_j$. Since $A_j$ is closed, the sets (in particular neighborhoods of points) that are open in $X\setminus A_j$, are open in $X$ as well. Hence $P$ is properly covered for the initial sequence $(A_1, \dots, A_k)$.
		
		To check that the set of properly covered points is everywhere dense in $A$, it suffices, for any open set $O$ that intersects $A$, to apply the proven statement to the space $O$ and the sequence $(O\cap A_1, O\cap A_2, \dots, O\cap A_k)$. 
		
		The set of properly covered points is open because some open neighborhood (in~$A$) of any properly covered point consists of properly covered points due to the definition. 
	\end{proof}
	
	\begin{lemma}
	\label{lem:compact-non-convex}
		Let $K$ be a compact subset of a Euclidean space~$E$. If $K$ is not convex, then $E$ contains a nondegenerate segment that intersects $K$ only by its endpoints. 
	\end{lemma}
	
	\begin{proof}
		Since $K$ is compact and not convex, it follows that the relative interior of the convex hull of $K$ contains a point $P$ such that $P \notin K$. By compactness argument it follows that $E$ contains a closed Euclidean ball $B$ centered at $P$ and such that $K \cap \Int B = \emptyset$ while $\partial{B}$ intersects~$K$. Let $Q$ be a point in $B\cap K$, and let $H$ be the hyperplane in $E$ that touches $B$ at~$Q$. Denote by $H^+$ the closed half-space bounded by $H$ and containing~$B$.		
		Since $\Int H^+$ contains $P$ and $P$ lies in the convex hull of $K$, it follows that $\Int H^+$ contains points of $K$ as well.		
		Let $T$ be a point in $K \cap \Int H^+$, and let $I$ be the segment with endpoints $Q$ and $T$. Observe that some punctured neighborhood of $Q$ in $I$ lies in $\Int B$ and hence does not intersect~$K$. Then, since $K$ is compact, it follows that $I$ contains a segment with desirable properties. 		
	\end{proof}

	\begin{theorem} 
\label{th:vis-is-infinite}	
	Let $M$ be a compact Riemannian manifold, let $d\colon M\times M\to[0,\infty)$ be any intrinsic metric on~$M$ compatible with its  topology, and let $f\colon M \to \mathbb{R}^{m}$ be a continuous map. Then $\Omega_f^{vis}$ for $(M,d)$ is infinite.
	%Let $n$ and $m$ be positive integers, let $M$ be a compact Riemannian manifold of dimension $n$, and let $f\colon M \to \mathbb{R}^{m}$ be a continuous map. Then $\Omega_f^{vis}$ is infinite.
	\end{theorem}
	
%	\begin{remark} 
%		It can be shown via a minor refinement of the proof of Theorem~\ref{th:vis-is-infinite} that Theorem~\ref{th:vis-is-infinite} is true for any intrinsic metric on~$M$ compatible with its  topology.
%	\end{remark}	
	
	\begin{proof}
Let $n=\dim M$ be the dimension of~$M$.
		The case where $m \leq n$ is covered by the Hopf theorem. We study the case where $n < m$ and prove the theorem by case analysis.
		%examining all possible cases. 
		%We use $d$ for the metric on~$M$.
		
		\begin{enumerate}
			\item If $f$ is not locally injective, then the statement of the theorem follows by Proposition~\ref{prop:non-locally-injective} because $f$-neighbors are visual $f$-neighbors.
			 
			\item If $f$ is locally injective, let $P$ be any point in~$f(M)$. Since $f$ is locally injective and $M$ is compact, it follows that $f^{-1}(P)$ is finite (see Proposition~\ref{prop:locally-injective-properties}(1)). Suppose $f^{-1}(P)$ consists of~$k$ points and denote them by $x_1, \dots, x_k$. Due to the assumed local injectivity of $f$, each $x_i$ has an open neighborhood $U_i$ such that the restriction of~$f$ to $U_{i}$ is injective (cf.~the above proof of assertion~(2) of Proposition~\ref{prop:locally-injective-properties}). Taking (if necessary) smaller neighborhoods of $x_1$, $\dots$, $x_k$, we can assume that $U_1$, $\dots$, $U_k$ are pairwise disjoint. Let $U'_i$ be an open neighborhood of $x_i$ such that the closure~$\overline{U'_i}$ of~$U'_i$ is contained in~$U_i$. 
			%(for illustration one can imagine $\overline{U'_i}$ as a topological ball or a metric ball in $M$, but we will not need this in our proof formally). 
			Observe that the complement $\mathcal{C}:=M\setminus(U'_1\cup \dots \cup U'_k)$ is compact and~$P\notin f(\mathcal{C})$. Hence, since a continuous image of a compact set is compact, it follows that there is a neighborhood $W = W(P)$ of~$P$ in~$\mathbb{R}^m$ such that $f(\mathcal{C})\cap W(P) = \emptyset$.
%			$$(f(M\setminus(U'_1\cup \dots \cup U'_k)) \cap W(P) = \emptyset.$$ 
Thus the intersection of $W(P)$ with $f(M)$ is the union of intersections of $W(P)$ with compact sets $f(\overline{U'_1})$, \dots, $f(\overline{U'_k})$. 
			By Lemma~\ref{lem:properly-covered} there exists an open set $V \subset W(P)$ such that $V\cap f(M)$ is nonempty and all nonempty sets in the sequence $(V\cap f(\overline {U'_1})$, $\dots$, $V\cap f(\overline{U'_k}))$ coincide. Let $D'$ be a closed Euclidean ball in $\mathbb{R}^{m}$ such that $D' \subset V$ and $f(M)\cap\Int D' \neq \emptyset$. We have two subcases:
			%In what follows we will discuss two subcases:
			
			\begin{enumerate}
			\item If $f(M)$ has arbitrarily close visual neighbors (see Definition~\ref{def:visual neighbors}) lying in~$D'$, then, since due to the choice of $V$ all nonempty sets in the sequence $(D'\cap f(\overline{U'_1}), \dots, D'\cap f(\overline {U'_k}))$ coincide with $D'\cap f(M)$, we conclude that in each $U'_i$ with nonempty $D'\cap f(\overline{U'_i})$ there is an infinite sequence of pairs of visual $f$\nobreakdash-neighbors $(a_j, b_j)$ such that the distance between $f(a_j)$ and $f(b_j)$ tends to $0$ as $j$ tends to infinity. Since the restriction of $f$ to $\overline{U'_i}$ is an embedding and $\overline{U'_i}$ is compact, we conclude that $d(a_j,b_j) \to 0$ when $j\to\infty$ as well. This proves the statement for this subcase. 
				
				\item If there exists $\delta > 0$ such that $D'\cap f(M)$ contains no visual neighbors at distance less than $\delta$ from each other, then take a closed Euclidean ball $D'' \subset D'$ of diameter less than $\delta$ and such that $\Int D''$ intersects~$f(M)$. Then $D''\cap f(M)$ is convex by Lemma~\ref{lem:compact-non-convex}. Let $H$ be the affine hull of $D''\cap f(M)$ in $\mathbb{R}^{m}$, and let $\dim H$ be its dimension (or, which is the same, the dimension of $D''\cap f(M)$). Observe that $\dim H = n = \dim M$. Indeed notice that due to the choice of $V$ (which contains $D'$ and $D''$) all nonempty sets in the sequence $(D''\cap f(\overline{U'_1}), \dots, D''\cap f(\overline{U'_k}))$ coincide with $D''\cap f(M)$ while for each $i\in\{1,\dots,k\}$ the restriction $f|_{\overline{U'_i}}$ of~$f$ to $\overline{U'_i}$ is an embedding of a compact set to a Hausdorff space and hence $f|_{\overline{U'_i}}$ is a homeomorphism between $\overline{U'_i}$ and $f(\overline{U'_i})$. Thus for an arbitrary point $x$ in the relative interior of $D''\cap f(M)$ the map $f$ gives a homeomorphism between some neighborhood of any point in $f^{-1}(x)$ and a neighborhood of $x$ in $D''\cap f(M)$. Then the desired equality of dimensions $\dim H = n$ follows by the domain invariance theorem. Now the proof splits into two subsubcases: 
				
				\begin{enumerate}
					\item If $H$ contains $f(M)$ then the statement readily follows from the Hopf theorem. (Supplemented by the fact that, due to compactness, a sequence of pairs of points arbitrarily close to each other with respect to the standard metric preserves this property when passing to any other metric compatible with the topology.)
					
					\item If $H$ does not contain $f(M)$, let $T$ be a point in $f(M) \setminus H$, let $H_+$ be the half-space of dimension $n+1$ bounded by $H$ and containing~$T$, and let $Q$ be an arbitrary point in the relative interior of $D''\cap f(M)$. Next let $\mathcal{B}$ be the family of closed $m$-dimensional Euclidean balls centered at points of $H_+$ and touching $H$ at $Q$. It is clear that balls in $\mathcal{B}$ with sufficiently small radii lie in $D''$ and intersect $f(M)$ just at $Q$, while balls in $\mathcal{B}$ with sufficiently large radii contain~$T$.
By compactness arguments and since a part of $f(M)$ that is close to $Q$ is contained in $H$, we conclude that there is a ball $B$ in $\mathcal{B}$ such that $f(M) \cap \Int B = \emptyset$ and $(\partial B \cap f(M)) \setminus \{Q\} \neq \emptyset$. Let $Q'$ be a point in $(\partial B \cap f(M)) \setminus\{Q\}$. It is easily seen by construction that all points of $f(M)$ in a small neighborhood of $Q$ are visual neighbors of~$Q'$. Now by virtue of the above-mentioned homeomorphisms between neighborhoods of points in $f^{-1}(Q)$ and a neighborhood of $Q$ in $D''\cap f(M)$ it follows that all points in some sufficiently small neighborhood of $f^{-1}(Q)$ are visual $f$-neighbors for all points in $f^{-1}(Q')$. 
Let $q$ be a point in $f^{-1}(Q)$, let $q'$ be a point in $f^{-1}(Q')$, and let $\rho$ be a shortest path joining $q$ and $q'$ (this~$\rho$ exists due to compactness of~$M$).
Then subpaths of $\rho$ are also shortest paths, which shows that in this subsubcase $\Omega_f^{vis}$ is continual for intrinsic metrics.
The theorem is proved. \qedhere
					
				\end{enumerate}
				
			\end{enumerate}

		\end{enumerate}

	\end{proof}

	To prove the result about spherical $f$-neighbors for the case of continuous maps $\mathbb{S}^1 \to \mathbb{R}^2$, we need several additional constructions and statements.
	
	Let $\mathbb{R}^2$ be the Euclidean plane. Suppose $\mathcal{O} \subset \mathbb{R}^2$ is a nonempty bounded open connected set.
	
	\begin{definition}
	\label{def:good-disk}
		A closed Euclidean disk $B(x) \subset \overline{\mathcal{O}}$ centered at $x \in \mathcal{O}$ is called a \emph{good disk} if $\card(B(x) \cap \partial{\mathcal{O}}) \geq 2$ (see Fig.~4). Points of $B(x) \cap \:\partial{\mathcal{O}}$ are called \emph{good points}. The union of good points corresponding to all good disks in $\overline{\mathcal{O}}$ we denote by $\mathcal{G}_{\mathcal{O}}$. Points of $\partial{\mathcal{O}} \setminus \mathcal{G}_{\mathcal{O}}$ are called \emph{bad points}.   
	\end{definition}
	
	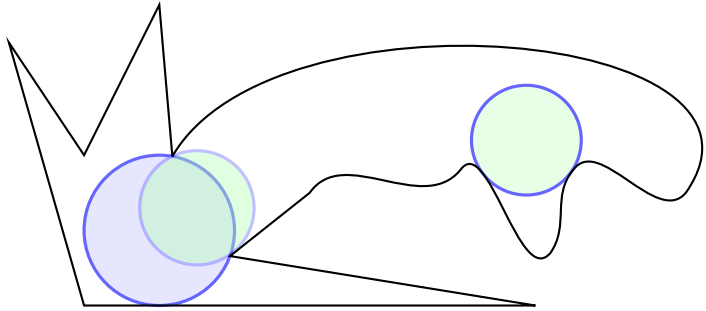
\begin{figure}[h]
		\centering
		\begin{tikzpicture}[>=latex]
			
			\filldraw[color=blue!60, fill=blue!10, very thick] (-6,0) circle (1);
			
			\filldraw[color=blue!60, fill=green!30, opacity = 0.4, very thick] (-5.5,0.3) circle (0.76);

			\filldraw[color=blue!60, fill=green!10, very thick] (-1.12,1.2) circle (0.73);
			
			\draw[thick, black, scale=1] ({-6 + 1*cos(80)},{0 + 1*sin(80)}) to [out=60, in=55] (1,0.5) to [out=230, in=55] (-0.5,0.8) to [out=230, in=55] (-0.8,-0.3) to [out=230, in=55] (-2,0.8) to [out=230, in=55] (-4,0.5) to ({-6 + 1*cos(-20)},{0 + 1*sin(-20)}) to (-1, -1) to (-7, -1) to (-8, 2.5) to (-7, 1) to (-6, 3) to ({-6 + 1*cos(80)},{0 + 1*sin(80)});
			
		\end{tikzpicture}
		\caption{An example of good disks in $\overline{\mathcal{O}}$. Notice that a good point can belong to uncountably many good disks.}
	\end{figure}
	
	For example, an ellipse in $\mathbb{R}^2$ has two bad points and a square has four bad points. For a ‘triangle’ whose sides are the Koch curve, bad points are everywhere dense in $\partial{\mathcal{O}}$. The following lemma shows that good points are always everywhere dense in $\partial{\mathcal{O}}$.
	
	\begin{lemma}
	\label{lem:good-points-dense}
		$\overline{{\mathcal{G}}_{\mathcal{O}}} = \partial{\mathcal{O}}$.
	\end{lemma}
	
	\begin{proof}
	 It suffices to show that if $b$ is a bad point in $\partial{\mathcal{O}}$ then 
	 a sequence of good points converges to~$b$.
	 %$b$ is an accumulation point for the set of good points in $\partial{\mathcal{O}}$.
		%Let $b$ be a bad point in $\partial{\mathcal{O}}$. 
		Take a sequence $(x_k)_{k=1}^{\infty}$ in $\mathcal{O}$ converging to~$b$. Note that, since $\mathcal{O}$ is the union of interiors of all good disks in $\overline{\mathcal{O}}$ (see~\cite{L04,M88,E29}), each $x_k$ lies in some good disk. Let $B_k$ be a good disk containing~$x_k$. Passing to a subsequence if necessary, we may assume that the sequence $(B_k)_{k=1}^{\infty}$ converges in the ‘disk space.’ Let $B_{\infty}$ be the limit of $(B_k)_{k=1}^{\infty}$. If $B_{\infty} = \{b\}$, then the statement of the lemma easily follows. Otherwise $B_{\infty}$ is nondegenerate. It is clear that $B_{\infty} \subset \overline{\mathcal{O}}$, because $B_k$ is in $\overline{\mathcal{O}}$ for each $k \in \mathbb{N}$. Next observe that $B_{\infty} \cap \overline{\mathcal{O}} = b$, since otherwise $b$ is not a bad point. Now for each good disk $B_k$, choose a good point $\xi_k\in B_k$. Any accumulation point of $\{\xi_k\}_{k=1}^{\infty}$ is a point in $\overline{\mathcal{O}}$ and a point in $\partial{B_{\infty}}$, i.\,e. it is the point $b$. This completes the proof. %\qedhere
	\end{proof}
		
		\begin{proposition}
	\label{prop:simply-connected}
		\begin{enumerate}
		\item Any open connected subset of~$\mathbb{R}^2$ is path-connected.
		\item If a subset~$K$ of $\mathbb{R}^2$ is closed and connected then each bounded component of the complement $\mathbb{R}^2\setminus K$ is simply connected.
		\item If $O$ is an open, connected, and simply connected subset of $\mathbb{R}^2$ and $H$ is a nondegenerate segment in $\mathbb{R}^2$ with relative interior in $O$ and with endpoints in $\partial{O}$, then the set $O\setminus H$ has precisely two connected components.
		\end{enumerate}
	\end{proposition}	
	
	\begin{proof}
	\begin{enumerate}
		\item This readily follows from the fact that~$\mathbb{R}^m$ is locally path-connected (i.\,e., it has a basis of path-connected neighbourhoods). Each open subset of a locally path-connected space is also locally path-connected. Each path-connected component of a locally path-connected space is open. Consequently, connected components of a locally path-connected space are the same as its path-connected components. 
			\item 
			Assume to the contrary that a bounded component~$O$ of $\mathbb{R}^2\setminus K$ is not simply connected. 
			Then there exists a loop $\gamma\colon [0,1]\to O$ noncontractible in~$O$.
			Since $\gamma([0,1])$ is compact and $O$ is open, it follows that $\gamma$ is homotopic in~$O$ to a polygonal loop in general position. 
			Then we see that there exists a simple noncontractible loop $\gamma_s\colon [0,1]\to O$. 
			This means that the region~$B$ in~$\mathbb{R}^2$ bounded by~$\gamma_s$ (here, we use the Jordan--Schoenflies theorem) is not contained in~$O$.
			Therefore, since $\partial O$ is contained in~$K$ and $K$ is connected, it follows that $B$ contains~$K$. This contradicts the assumption that $O$ is bounded.
			
		    \item We first show that $O\setminus H$ is disconnected. 
		    Assume to the contrary that $O\setminus H$ is connected. 
		    Let $I$ be a ‘small’ segment lying in $O$ and intersecting~$H$ transversely. 
		    Since $O\setminus H$ is open and assumed to be connected, it follows that $O\setminus H$ is also path-connected (see the first assertion of the proposition).
		    Therefore, there exists a path $\rho$ lying in $O\setminus H$ and joining the endpoints of~$I$.
			Then, passing to isotopic polygonal lines as in the proof of the previous assertion, it is easy to deduce that there exists a simple closed polygonal line~$\gamma$ lying in $O$ and intersecting~$H$ transversely at a single point ($H\cap I$). This implies that the endpoints of~$H$ lie in distinct components of the set $\mathbb{R}^2\setminus \gamma$ (use the Jordan--Schoenflies theorem), which contradicts the fact that~$O$ is simply connected since these endpoints are both in~$\partial{O}$. This contradiction shows that $O\setminus H$ is disconnected. 
		
		To verify that $O\setminus H$ has at most two components, we construct an open neighborhood $U\subset O$ of the relative interior of~$H$ such that $U$ is the union of open Euclidean disks, each centered at a point of~$H$. Then $U\setminus H$ is clearly a two-component set while each component of $O\setminus H$ intersects (and hence contains) a component of $U\setminus H$.		\qedhere
		\end{enumerate}
	\end{proof}
	
	\begin{definition}
	\label{def:ruled}
		We call a nondegenerate segment in $\overline{\mathcal{O}}$, with relative interior in $\mathcal{O}$ and with endpoints in $\partial{\mathcal{O}}$, a \emph{chord}. 
		We say that a chord is \emph{good} if it is contained in a good disk.		
		A sequence of chords $(H_k)_{k=1}^{\infty}$ in $\overline{\mathcal{O}}$ is called \emph{ruled} if for each $k > 1$, the chord $H_k$ divides $\mathcal{O}$ in two parts such that one of these parts contains the relative interior of $H_{k-1}$ and another one contains the relative interior of $H_{k+1}$.
	\end{definition}
	
\begin{remark}
Definitions readily imply that two points in $\partial{\mathcal{O}}$ are spherical neighbors with respect to $\partial{\mathcal{O}}$ if and only if these two points are the endpoints of a good chord.
\end{remark}

	\begin{lemma}
	\label{lem:ruled}
		If $\mathcal{O}$ is simply connected then there is a ruled sequence of good chords in $\overline{\mathcal{O}}$.
	\end{lemma}
	
	\begin{proof}
Observe that $\overline{\mathcal{O}}$ has at least one good chord because $\overline{\mathcal{O}}$ contains at least one good disk (\cite{L04,M88,E29}) while the segment joining two good points on the boundary of this good disk is a good chord. Next observe that in order to prove the lemma it suffices to show that if $H$ is a good chord then the set $\mathcal{O} \setminus H$ is composed of two connected components and each of these components contains the relative interior of a good chord (of~$\overline{\mathcal{O}}$). The set $\mathcal{O} \setminus H$ is composed of two connected components because $\mathcal{O}$ is simply connected (see Proposition~\ref{prop:simply-connected}). Now, if $R$ is a component of $\mathcal{O} \setminus H$, let $B$ be that of all good disks containing~$H$ which has maximal intersection with~$R$. 	If~$R$ is contained in~$B$ then the part of $\partial \mathcal{O}$ that bounds~$R$ is an arc of the circle~$\partial B$, and hence $R$ contains relative interiors of a continuum of good chords. If~$R$ is not contained in~$B$, let $x$ be a point in~$R\setminus B$. Then there exists a good disk~$D$ containing~$x$ (see~\cite{L04,M88,E29}). Clearly, the relative interior of any good chord contained in~$D$, is contained in~$R$. The lemma is thus proved.
	\end{proof}
	
	\begin{definition}
	\label{def:locally-simple}
	Let $\mathbb{S}^1$ be a circle, and let $f\colon\mathbb{S}^1 \to \mathbb{R}^2$ be a continuous map. If $f$ is locally injective (see Definition~\ref{def:locally-injective}) we say that $f$ is a \emph{locally-simple curve}. If the restriction of $f$ to an arc $\alpha\subset\mathbb{S}^1$ is injective, we say that $\alpha$ is an \emph{$f$-simple arc}.
	\end{definition}

%	\begin{convention}
%	\label{convention:left}
%	\end{convention}

	The following proposition lists several properties of locally-simple curves.
	% and prove new ones. % ((3) and (4)).
	
	\begin{proposition}
	\label{prop:locally-simple-properties}
		If $f\colon\mathbb{S}^1 \to \mathbb{R}^2$ is a locally-simple curve, then 
		\begin{enumerate}
			\item For any $y\in f(\mathbb{S}^1)$, the inverse image $f^{-1}(y)$ is finite;  
			\item Each point $y\in f(\mathbb{S}^1)$ has a neighborhood $U=U(y)$ such that each connected component of $f^{-1}(U)$ is an $f$-simple arc;
			\item $\Int f(\mathbb{S}^1) = \emptyset$;
			\item The set of bounded connected components of\:~$\mathbb{R}^2 \setminus f(\mathbb{S}^1)$ is nonempty;
			\item The set of indices (the winding numbers) of connected components of\:~$\mathbb{R}^2 \setminus f(\mathbb{S}^1)$ is finite. 
		\end{enumerate}
	\end{proposition}
	
	\begin{proof}
		%We observe that 
		\begin{enumerate}
			\item This is a particular case of assertion~(1) of Proposition~\ref{prop:locally-injective-properties}.
			% case (1) in the above proof of Theorem~\ref{th:vis-is-infinite}. 
			%Indeed if $f^{-1}(b)$ is infinite take a converging sequence of points in $f^{-1}(b)$. The limit point of this sequence is not contained in any $f$-simple arc.
			
			\item This follows from assertion~(2) of Proposition~\ref{prop:locally-injective-properties}. (If $f$ is injective and we have $f^{-1}(U)=\mathbb{S}^1$, we pass to a smaller neighborhood.)
			
			\item This follows, for example, from the domain invariance theorem.
			
			%from the well-known result that there is not a homeomorphism between $f$-simple arc and a set with non-empty interior in $\mathbb{R}^2$ (for example, take a preimage of a point in $\Int f(\mathbb{S}^1)$, it is a cut-point of $\mathbb{S}^1$, but not of $\Int f(\mathbb{S}^1)$).
			
			\item The statement follows, for example, from Theorem 2.1 in \cite{W47}. 
						
			\item Since all of the components of $\mathbb{R}^2 \setminus f(\mathbb{S}^1)$ having nonzero index are contained in a compact domain, it follows by standard compactness arguments that in order to prove the statement it suffices to show that an arbitrary point $z\in \mathbb{R}^2$ has a neighborhood $U(z)$ such that the set of indices (the winding numbers) of points in $U(z)\setminus f(\mathbb{S}^1)$ is finite. 
			If $z\notin f(\mathbb{S}^1)$ then $z$ has a neighborhood where all points are of the same index.
			Let $z$ be a point of~$f(\mathbb{S}^1)$.
			By assertion~(2) %of Proposition~\ref{prop:locally-injective-properties}, 
			there exists $\varepsilon_0 > 0$ such that all connected components of $V_{\varepsilon_0}:= f^{-1}(D_{\varepsilon_0}(z))$, where $D_{\varepsilon_0}(z)$ is the open disk of radius $\varepsilon_0$ centered at~$z$, are (open) $f$-simple arcs.
%			
%the restriction of~$f$ to each connected component of $V_{\varepsilon_0}:= f^{-1}(D_{\varepsilon_0}(b))$, where $D_{\varepsilon_0}(b)$ is the open disk of radius $\varepsilon_0$ centered at~$b$, is injective.
%			
By assertion~(1), only a finite number of these $f$-simple arcs of~$V_{\varepsilon_0}$ intersect $f^{-1}(z)$. Let $\gamma_1$, \dots, $\gamma_k$ be the closures of the $f$-simple arcs of~$V_{\varepsilon_0}$ intersecting $f^{-1}(z)$. Passing to a smaller~$\varepsilon_0$ if necessary, we may assume that for each $i\in\{1,\dots,k\}$, the closed arc $\gamma_i$ is $f$-simple as well and $f(\gamma_i)$ is a closed simple arc (not a loop) with two distinct endpoints on the circle $\partial D_{\varepsilon_0}(z)$.

Let $f'\colon \mathbb{S}^1 \to \mathbb{R}^2$ be a curve (not necessarily locally-simple) of the following form: 
\begin{itemize}
\item[--] for each $i\in\{1,\dots,k\}$, the restriction of $f'$ to $\gamma_i$ is injective and the image $f'(\gamma_i)$ is an arc of $\partial D_{\varepsilon_0}(z)$ with the same pair of endpoints as that of $f(\gamma_i)$; 
\item[--] for each $t\in \mathbb{S}^1 \setminus(\gamma_1\cup\dots\cup\gamma_k)$ we set $f'(t)=f(t)$.
\end{itemize}

Observe that for each $i\in\{1,\dots,k\}$, the union $f'(\gamma_i)\cup f(\gamma_i)$ is a Jordan curve. By the Jordan--Schoenflies theorem this curve bounds an open topological disk $D_i\subset D_{\varepsilon_0}(z)$. By construction, for any point $y\in D_{\varepsilon_0}(z)\setminus f(\mathbb{S}^1)$ the index of~$y$ with respect to~$f$ differs from the index of~$y$ with respect to~$f'$ by at most~$k$.

Next, observe that $z\notin f'(\mathbb{S}^1)$ so that by compactness of~$f'(\mathbb{S}^1)$ there exists $\varepsilon_1\in(0,\varepsilon_0)$ such that the open disk $D_{\varepsilon_1}(z)$ of radius $\varepsilon_1$ centered at~$z$ does not intersect~$f'(\mathbb{S}^1)$. Consequently, all of the points in $D_{\varepsilon_1}(z)\setminus f(\mathbb{S}^1)$ have the same index with respect to $f'$ and the set of their indices with respect to~$f$ has cardinality at most $k+1$. It remains to set $U(z):=D_{\varepsilon_1}(z)$, which completes the proof. \qedhere
		\end{enumerate} 
	\end{proof}

	\begin{theorem}
	\label{th:sph-is-infinite}
	Let~\,$\mathbb{S}^1$ be a Euclidean circle in~$\mathbb{R}^2$, let $d$ be the angular distance on~\,$\mathbb{S}^1$, and let $f\colon \mathbb{S}^1\to\mathbb{R}^2$ be a continuous map. Then $\Omega_f^{sph}$ for $(\mathbb{S}^1,d)$ is infinite.
	\end{theorem}
	
	\begin{proof}
	%\begin{enumerate}
	%	\item[Case~1:] $f$ is not locally-simple. In this case $\Omega_f^{sph}$ is infinite by Proposition~\ref{prop:non-locally-injective} because $f$\nobreakdash-neighbors are spherical $f$-neighbors.
	%	\item[Case~2:] $f$ is locally-simple.
	%	\end{enumerate}
	%We distinguish the case where $f$ is locally-simple and the case where $f$ is not locally-simple. 
	Our proof is a case-by-case analysis.
	On its first level we distinguish two basic cases, the case where $f$ is locally-simple and the case where it is not.
If $f$ is not locally-simple then $\Omega_f^{sph}$ is infinite by Proposition~\ref{prop:non-locally-injective} because $f$\nobreakdash-neighbors are spherical $f$-neighbors. 
	The rest of the proof addresses subcases of the case
	\begin{enumerate}
	\item[(C2)] $f$ is locally-simple.
	\end{enumerate}
		
The following exposition consists of two parts.
%In this case the proof will consist of two parts. 
First we give some preliminaries and introduce the notion of \emph{positive points}, and then we pass to the core construction of the proof, which involves studying sequences of pairs of positive points.
			
	We will use the classical notion of \emph{winding number} (or \emph{winding index}) of a closed curve in~$\mathbb{R}^2$ around a given point (which we also refer to as the index of a point with respect to a curve). The index of a point~$q$ with respect to a curve~$\gamma\colon\mathbb{S}^1\to\mathbb{R}^2$ will be denoted by $\wind(q,\gamma)$.
			The index depends on orientations.
			We fix orientations on~$\mathbb{S}^1$ and on~$\mathbb{R}^2$.
			For convenience of description we assume that~$\mathbb{R}^2$ is located ‘in front of us’ in $\mathbb{R}^3$ in such a way that any Euclidean circle (in our~$\mathbb{R}^2$) oriented counterclockwise has winding index~$1$ with respect to the points inside~it.
			Under this assumption, if we have, say, a smooth closed curve in general position then the indices of any two adjacent components of the curve complement differ by~$1$: the component with the larger index is on the \emph{left} side of the curve when we move along the curve according to its orientation. 
			
			%Another related notion we use is that of positive points, which we now define. 
			%If $A$ is a simple arc in $\mathbb{R}^2$ with its relative interior in~$\mathbb{R}^2\setminus f(\mathbb{S}^1)$ and $\xi\in f(\mathbb{S}^1)$ is an endpoint of~$A$, we say that a point $p\in f^{-1}(\xi)$ is \emph{positive} for~$A$ if for some 
			
			Next we introduce the notion of positive points.
			Let $H$ be a simple closed arc in~$\mathbb{R}^2$ such that the relative interior~$H^\circ$ of~$H$ does not intersect~$f(\mathbb{S}^1)$ and an endpoint $\xi$ of $H$ is in $f(\mathbb{S}^1)$.
			Let $x$ be a point in~$f^{-1}(\xi)$. 
			%Let $O$ be a component of ${\mathbb{R}^2\setminus f(\mathbb{S}^1)}$, let $H$ be a chord in $\overline{{O}}$ (see Definition~\ref{def:ruled}), let $\xi\in \partial{{O}}$ be an endpoint of~$H$, and let $x$ be a point in~$f^{-1}(\xi)$. 
			Since $f$ is locally-simple, it follows that there exists a closed $f$-simple arc~$\gamma$ containing~$x$ in its interior.
			The Jordan--Schoenflies theorem implies~\cite{Th11} that there exists an orientation preserving homeomorphism $\Psi\colon \mathbb{R}^2\to\mathbb{R}^2$ such that $\Psi(f(\gamma))$ is a straight segment. 
		We say that $x$ is a \emph{positive point} for~$H$ 
		%(and that $f(\gamma)$ is a \emph{positive arc} for~$H$) 
		if the arc $\Psi(H)$ adjoins our segment $\Psi(f(\gamma))$ on the left side when we move along $\Psi(f(\gamma))$ according to its orientation (induced by that of~$f(\gamma)$). 
		%(In what follows, in analogous situations we say that $\Psi(H)$ adjoins $\Psi(f(\gamma))$ on the left side with respect to a prescribed orientation.)
		It is straightforward to check by studying compositions of homeomorphism of $\mathbb{R}^2$ that, for $x$ and~$H$, this definition is independent of choices of $\Psi$ and~$\gamma$.
		%depends neither on the choice of $\Psi$ nor on the choice of~$\gamma$.
			
		In the above notation, let ${O}_H$ be the component of $\mathbb{R}^2\setminus f(\mathbb{S}^1)$ containing~$H^\circ$. 
		We claim that if ${O}_H$ is a component where the index attains its largest value then at least one point in $f^{-1}(\xi)$ is positive for~$H$. 
		%We claim that if, in the above notation, ${O}$ is a component where the index attains its largest value then at least one point in $f^{-1}(\xi)$ is positive for~$H$. 	
			
		To prove this claim, we need an auxiliary construction. Assertion~(2) of Proposition~\ref{prop:locally-simple-properties} implies that there exists $\varepsilon > 0$ such that the inverse image $V_{\varepsilon,\xi}= f^{-1}(D_{\varepsilon}(\xi))$, where $D_{\varepsilon}(\xi)$ is an open disk of radius $\varepsilon$ centered at $\xi$, is an (at most countable) collection of open $f$-simple arcs. By assertion~(1) of Proposition~\ref{prop:locally-simple-properties}, $f^{-1}(\xi)$ is finite so that only a finite number of these arcs intersect $f^{-1}(\xi)$. Let $k$ be the cardinality of $f^{-1}(\xi)$, let $x_1$, $\dots$, $x_k$ be the points of~$f^{-1}(\xi)$, and let $\gamma_1$, \dots, $\gamma_k$, where $\gamma_i\ni x_i$, be the closures of $f$-simple arcs of~$V_{\varepsilon,\xi}$ intersecting $f^{-1}(\xi)$. Passing to a smaller~$\varepsilon$ if necessary, we may assume that for each $i\in\{1,\dots,k\}$, the closed arc~$\gamma_i$ is $f$-simple as well and $f(\gamma_i)$ is a closed simple arc (not a loop) with two distinct endpoints on the circle $\partial D_{\varepsilon}(\xi)$. 	
		Observe that for each $i\in\{1,\dots,k\}$, $f(\gamma_i)$ splits $D_{\varepsilon}(\xi)$ into two parts, each of which is homeomorphic to $D_{\varepsilon}(\xi)$ by the Jordan--Schoenflies theorem. 
		%the set $D_{\varepsilon}(\xi)\setminus f(\gamma_i)$ 
		Let $D_i$ denote that one of these two parts which does not intersect~$H$, and let $J_i$ be the closed arc $\partial D_i \cap \partial D_{\varepsilon}(\xi)$ ($J_i$~is an arc of $\partial D_{\varepsilon}(\xi)$ and has the same endpoints as $f(\gamma_i)$).	

		Set $\delta_i=-1$ if $x_i$ is positive for~$H$ and $\delta_i=1$ otherwise.
		It is straightforward to check that if $f_1\colon \mathbb{S}^1 \to \mathbb{R}^2$ is a curve such that 
		$f_1(\gamma_1)=J_1$ and $f_1(t)=f(t)$ for $t\in \mathbb{S}^1\setminus\gamma_1$,
		then for an arbitrary point $q\in\mathbb{R}^2\setminus (f(\mathbb{S}^1)\cup f_1(\mathbb{S}^1))$ we have 
$$
\wind(q,f_1)=
\begin{cases}
\wind(q,f)+\delta_1 &\text{if}~ q\in D_1,\\
\wind(q,f) &\text{if}~ q\notin D_1.
\end{cases}
$$		
%$\wind(q,f_1)=\wind(q,f)$ if $q\notin D_1$ and $\wind(q,f_1)=\wind(q,f)+\delta_1$ if $q\in D_1$.
		
Applying this argument to all arcs $f(\gamma_i)$ we see that if $p$ is a point in~${O}_H$ and $q$ is a point in $\mathbb{R}^2\setminus f(\mathbb{S}^1)$ such that $p$ and $q$ are joined by a path that intersects $f(\mathbb{S}^1)$ only at points of $f(\gamma_1\cup\dots\cup\gamma_k)$ then 
\begin{equation}
\label{eq:wind}
\wind(q,f)=
\wind(p,f)+\sum_{\{i \mid  q\in D_i\}} \delta_i.
\end{equation}

Observe that if $\varepsilon_1 > 0$ is sufficiently small then the open disk~$D_{\varepsilon_1}(\xi)$ of radius~$\varepsilon_1$ centered at $\xi$ intersects~$O_H$ and~$D_1$ and does not intersect $f(\mathbb{S}^1)\setminus f(\gamma_1\cup\dots\cup\gamma_k)$. 
Then it follows by~\eqref{eq:wind} in an obvious way that if ${O}_H$ is a component where the index attains its largest value then at least one of $\delta_i$'s is negative so that 
at least one of $x_i$'s is positive for~$H$.
The claim is thus proved.
			
			We pass to the second part of the proof for case~(C2). 
			
			Assertions (4) and~(5) of Proposition~\ref{prop:locally-simple-properties} imply that there exists a bounded connected component of $\mathbb{R}^2 \setminus f(\mathbb{S}^1)$ where the index attains its extreme value. 
			Let~$\mathcal{O}$ be such a component.
			%Denote any such component by~$\mathcal{O}$. 
			By reversing the orientation of~$\mathbb{S}^1$ if necessary we get that in~$\mathcal{O}$ the index attains the largest value.\footnote{Notice that if our curve with its initial orientation is, e.\,g., a Euclidean circle oriented clockwise, then there is no bounded connected component of $\mathbb{R}^2 \setminus f(\mathbb{S}^1)$ where the index attains its largest value (which is $0$ in this case).}
			
			Observe that $\mathcal{O}$ is simply connected because it is a bounded connected component of $\mathbb{R}^2 \setminus f(\mathbb{S}^1)$ while $f(\mathbb{S}^1)$ is compact and connected (see Proposition~\ref{prop:simply-connected}).
			Then Lemma~\ref{lem:ruled} implies that there exists a ruled sequence of good chords in $\overline{\mathcal{O}}$.
			Let $(H_i)_{i=1}^\infty$ be such a sequence.
			%Fix an arbitrary orientation for each~$H_i$.
			%Let $\xi_i$ denote the initial, with respect to the orientation, endpoint of~$H_i$
			For each~$i$, let $\xi_i$ and $\zeta_i$ be the endpoints of~$H_i$ listed in an arbitrary order.
			%For each~$i$, choose an endpoint of~$H_i$, denote it by $\xi_i$, and let $\zeta_i$ be the other endpoint.
			%It is shown in the above part of the proof 
			The claim in the above part of the proof implies
			that if $H$ is a chord for~$\overline{\mathcal{O}}$ and $\xi$ is an endpoint of~$H$, then at least one point in $f^{-1}(\xi)$ is positive for~$H$.
			For each~$i$, let $a_i\in f^{-1}(\xi_i)$ and $b_i\in f^{-1}(\zeta_i)$ be any points positive for~$H_i$.
			(Note that $a_i$ and $b_i$ are spherical $f$-neighbors because $\xi_i$ and $\zeta_i$ are the endpoints of a good chord.)
			
			Passing to subsequences if necessary, we may and will assume that each of the sequences $(a_i)_{i=1}^\infty$ and $(b_i)_{i=1}^\infty$ is 
			\begin{itemize}
			\item[(A1)] convergent (since $\mathbb{S}^1$ is compact),
			\item[(A2)] contained in a closed $f$-simple arc (since $f$ is assumed to be locally-simple),
			\item[(A3)] monotone in a containing $f$-simple arc (since each infinite numerical sequence contains an infinite monotone subsequence).
			\end{itemize}
(The correctness of transition to subsequences is ensured by the fact that subsequences of ruled sequences are ruled as well.)

Let $J_a\subset\mathbb{S}^1$ and $J_b\subset\mathbb{S}^1$ be closed $f$-simple arcs containing $(a_i)_{i=1}^\infty$ and $(b_i)_{i=1}^\infty$ respectively and such that $(a_i)_{i=1}^\infty$ is monotone in~$J_a$ and $(b_i)_{i=1}^\infty$ is monotone in~$J_b$.

We distinguish two subcases of (C2):
		\begin{enumerate}
		\item[(C2.1)] Either $(a_i)_{i=1}^\infty$ or $(b_i)_{i=1}^\infty$ is eventually constant. 
		\item[(C2.2)] None of $(a_i)_{i=1}^\infty$ and $(b_i)_{i=1}^\infty$ is eventually constant.
		\end{enumerate} 
 
If one of $(a_i)_{i=1}^\infty$ and $(b_i)_{i=1}^\infty$ is eventually constant then the second one is not (because $(H_i)_{i=1}^\infty$ is not eventually constant by construction), whence it follows that $\{d(a_i,b_i)\}_{i=1}^\infty$ is infinite. 
%(since $d$ is intrinsic while no `metric circle' in $\mathbb{S}^1$ with intrinsic metric contains more than two points). 
Then $\Omega_f^{sph}$ is infinite since $\Omega_f^{sph}$ contains $\{d(a_i,b_i)\}_{i=1}^\infty$ because $a_i$ and $b_i$ are spherical $f$-neighbors.

The rest of the proof deals with subcase~(C2.2).

In this subcase, since any monotone sequence is either eventually constant or contains a strictly monotone subsequence, passing again to subsequences if necessary we may and will assume in addition that
			\begin{itemize}
			\item[(A4$_a$)] $(a_i)_{i=1}^\infty$ is strictly monotone in $J_a$, 
			\item[(A4$_b$)] $(b_i)_{i=1}^\infty$ is strictly monotone in~$J_b$.
			\end{itemize}

Now, we have two subsubcases of (C2.2):
		\begin{enumerate}
		\item[(C2.2.1)] either $(a_i)_{i=1}^\infty$ is strictly increasing (in~$J_a$ with respect to the fixed  orientation of~$\mathbb{S}^1$) and $(b_i)_{i=1}^\infty$ is strictly decreasing (in~$J_b$ with respect to the fixed  orientation of~$\mathbb{S}^1$) or vice versa,
		\item[(C2.2.2)] $(a_i)_{i=1}^\infty$ and $(b_i)_{i=1}^\infty$ are either both strictly increasing or both strictly decreasing.
		\end{enumerate} 

		In case (C2.2.1) a straightforward geometric analysis shows that $\{d(a_i,b_i)\}_{i=1}^\infty$ is infinite. Then $\Omega_f^{sph}$ is infinite since $\Omega_f^{sph}$ contains $\{d(a_i,b_i)\}_{i=1}^\infty$ because $a_i$ and $b_i$ are spherical $f$-neighbors.

		Thus our proof boils down to the very specific subcase (C2.2.2). We will show that this subcase is ‘empty,’ i.\,e., no continuous map $f\colon \mathbb{S}^1\to\mathbb{R}^2$ satisfies all conditions determining this subcase. Denote by $\mathcal{A}$ the closed subarc of the simple arc $f(J_a)$ with endpoints at~$f(a_1)$ and $f(a_3)$, and let $\mathcal{B}$ be the closed subarc of $f(J_b)$ with endpoints at~$f(b_1)$ and $f(b_3)$.
		
		We split (C2.2.2) into two subcases: 
		\begin{enumerate}
		\item[(C2.2.2.1)] $\mathcal{A} \cap \mathcal{B} = \emptyset$,
		\item[(C2.2.2.2)] $\mathcal{A} \cap \mathcal{B} \neq \emptyset$.
		\end{enumerate} 
		
	In the first subcase (C2.2.2.1) observe that the set $$Q(\mathcal{A}, \mathcal{B}) = \mathcal{A} \cup \mathcal{B} \cup H_1 \cup H_3$$ is a Jordan curve and hence by the Jordan--Schoenflies theorem there is a homeomorphism $\psi\colon\mathbb{R}^2 \to \mathbb{R}^2$ that maps $Q(\mathcal{A}, \mathcal{B})$ to a Euclidean circle $\mathbb{S}_0^1$. 
	%Let $\mathcal{A}' = \psi(\mathcal{A})$ and $\mathcal{B}' = \psi(\mathcal{B})$. %Choose an orientation on $\mathbb{S}_0^1$. 
	 
	Let $r_a$ and $r_b$ denote the orientations of $\psi(\mathcal{A})$ and $\psi(\mathcal{B})$ determined by the orderings/directions $\psi(f(a_1))\to \psi(f(a_2))\to \psi(f(a_3))$ and $\psi(f(b_1))\to \psi(f(b_2))\to \psi(f(b_3))$ respectively.	 
	
	Then $r_a$ and $r_b$ induce opposite orientations on~$\mathbb{S}^1_0$. This means in particular that 
	\begin{itemize}
	\item[(P1)] if $X$ is a component of~$\mathbb{R}^2\setminus \mathbb{S}^1_0$ then either 
	$X$ adjoins $\psi(\mathcal{A})$ on the left side when we move along $\psi(\mathcal{A})$ according to $r_a$ and $X$ adjoins $\psi(\mathcal{B})$ on the right side when we move along $\psi(\mathcal{B})$ according to $r_b$, or vice versa.
	%
	%$X$ adjoins $\psi(\mathcal{A})$ on the left side with respect to~$r_a$ and $X$ adjoins $\psi(\mathcal{B})$ on the right side with respect to~$r_b$, or vice versa.
	\end{itemize}
	
	At the same time, since we address subcase (C2.2.2), where $(a_i)_{i=1}^\infty$ and $(b_i)_{i=1}^\infty$ are either both strictly increasing or both strictly decreasing, it follows that $r_a$ and~$r_b$ induce, via the map $\psi\circ f$, one and the same orientation on~$\mathbb{S}^1$.	
	Therefore, since $a_2$ and $b_2$ are positive for~$H_2$ by construction, it follows by the definition of positive points that 
	\begin{itemize}
	\item[(P2)]
	either $\psi(H_2)$ adjoins both $\psi(\mathcal{A})$ and $\psi(\mathcal{B})$ on the left side 
	%when we move along $\psi(\mathcal{A})$ and $\psi(\mathcal{B})$ according to $r_a$ and $r_b$,
	with respect to $r_a$ and $r_b$,
	or $\psi(H_2)$ adjoins both $\psi(\mathcal{A})$ and $\psi(\mathcal{B})$ on the right side with respect to $r_a$ and $r_b$.
	\end{itemize}
	
	However this is impossible: conditions~(P1) and~(P2) are incompatible because $\psi(H_2)$ intersects only one component of~$\mathbb{R}^2\setminus \mathbb{S}^1_0$.
	
	This shows that subcase (C2.2.2.1) is ‘empty.’
	
	In the second subcase (C2.2.2.2) where $\mathcal{A}$ and $\mathcal{B}$ intersect and it is possible that no Jordan curve contains them, we will study two Jordan curves. Let $m_i$ denote the middle point of $H_i$, and let $\tau$ be a simple polygonal path in~$\mathcal{O}$ with endpoints $m_1$ and $m_3$ and such that $\tau\cap H_i=\{m_i\}$ for $i\in\{1,2,3\}$.  %??? existence
	Let $H_i^a$ denote the half of $H_i$ joining $f(a_i)$ and~$m_i$, and let $H_i^b=\overline{H_i\setminus H_i^a}$.
	Observe that the sets 
	$$Q(\mathcal{A}, \tau) = \mathcal{A} \cup  \tau \cup H_1^a \cup H_3^a$$ 
	and 
	$$Q(\tau, \mathcal{B}) = \tau \cup \mathcal{B} \cup H_1^b \cup H_3^b$$ 
	are Jordan curves.
	Applying argument of subcase (C2.2.2.1) to these Jordan curves and using their common arc~$\tau$ we arrive at a similar set of incompatible conditions (details are left to the reader), which shows that subcase (C2.2.2.2) is ‘empty’ as well. 
This completes the proof of Theorem~\ref{th:sph-is-infinite}.	
	\end{proof}
	
	\begin{corollary}
	\label{cor:sph-is-infinite}
Let $S^1$ be a topological circle, let $d$ be an intrinsic metric on~$S^1$ compatible with its topology, and let $f\colon S^1\to\mathbb{R}^2$ be a continuous map. Then $\Omega_f^{sph}$ for $(S^1,d)$ is infinite.
	\end{corollary}
	
	\begin{proof}
	This follows from Theorem~\ref{th:sph-is-infinite} because any intrinsic metric on~$S^1$ is proportional to the metric induced by the angular metric on a Euclidean circle~$\mathbb{S}^1\subset \mathbb{R}^2$ via a homeomorphism~$S^1\to \mathbb{S}^1$. 
%Therefore, it is enough to prove the theorem for the angular metric on a Euclidean circle~$\mathbb{S}^1$.	
	\end{proof}
	
	%\begin{remark}
	%	Observe that Theorem~\ref{th:sph-is-infinite} is true not only for intrinsic metrics, but also for a wider class of metrics. Namely if a given metric is compatible with the standard topology on $\mathbb{S}^1$ and for any ruled sequence in the Euclidean disk $D^1$ ($\partial{D^1} = \mathbb{S}^1$), there exists a ruled subsequence such that the metric is monotone on endpoints of chords of this subsequence, then this metric is suitable.  
	%\end{remark}
	
	\section{A quantitative generalisation of the Hopf theorem}
	
	Let $n$ be a positive integer. The Hopf theorem states that if $M$ is a compact Riemannian manifold of dimension $n$ and $f\colon M \to \mathbb{R}^n$ is a continuous map, then for an arbitrary $\delta > 0$, there are points $a$ and $b$ in $M$ such that they are joined by a geodesic of length $\delta$ and $f(a) = f(b)$. 
	
	In this section we prove that if in addition to the assumptions of the Hopf theorem we assume that $n>1$ and no two points in $M$ (not necessarily distinct) are joined by an infinite number of geodesics of length $\delta > 0$, then $M \times M$ contains uncountably many pairs of points such that $f$ takes each pair to a singleton and points in each pair are joined by a geodesic of length~$\delta$.
	
	First we recall Hopf's original proof of his theorem and then we prove our theorem using the main idea of Hopf's proof.
	
	\begin{proof}		
		Without loss of generality we can suppose that in $M$ there are no closed geodescis of length $\delta$.
		
		Let $p \in M$, and let $T_pM$ be the tangent space at $p$. Choose a standard basis in $T_pM$, and let $\mathbb{S}_{p}^{n-1} \subset T_pM$ be the unit sphere. It is well known that for a given tangent vector $\mathfrak{F} \in T_pM$, there exists a unique geodesic with the tangent vector $\mathfrak{F}$ at $p$. Thus for each $\mathfrak{F} \in \mathbb{S}_{p}^{n-1}$, there is a unique geodesic, which we denote by $\gamma_{\mathfrak{F}}$.
		
		We will denote by $\gamma_{\mathfrak{F}}(s)$, where $s \in \mathbb{R}$, the geodesic $\gamma_{\mathfrak{F}}$ with a natural \sloppy parametrization, and we will suppose that $\gamma_{\mathfrak{F}}(0) = p$ and $\frac{d}{ds}(\gamma_{\mathfrak{F}}(s))|_{s=0} = \mathfrak{F}$ for each $\mathfrak{F} \in \mathbb{S}_{p}^{n-1}$. Observe that $\gamma_{\mathfrak{F}}$ and $\gamma_{\mathfrak{-F}}$ are the same geodesics, but $\gamma_{\mathfrak{F}}(s)$ and $\gamma_{\mathfrak{-F}}(s)$ differ in the direction of moving when $s$ increases. 
		
		Study the vector $V(\mathfrak{F})_{s, p, \delta} := f(\gamma_{\mathfrak{F}}(s + \delta)) - f(\gamma_{\mathfrak{F}}(s))$. Suppose that the statement of Theorem 1 is not true, and let $v(\mathfrak{F})_{s, p, \delta} = \frac{V(\mathfrak{F})_{s, p, \delta}}{|V(\mathfrak{F})_{s, p, \delta}|}$, where $|V(\mathfrak{F})_{s, p, \delta}|$ denotes the Euclidean norm of $V(\mathfrak{F})_{s, p, \delta}$ in $\mathbb{R}^n$. We get continuous maps $T_{s,p,\delta} \colon \mathbb{S}_p^{n-1} \to \mathbb{S}^{n-1}$ for each $p \in M$ and $s \in \mathbb{R}$, where $\mathbb{S}^{n-1}$ is the Euclidean sphere. Note that $T_{-\delta/2,p,\delta}$ is antipodal and hence the modulo 2 degree of $T_{-\delta/2,p,\delta}$ is 1 for each $p \in M$.
		
		Denote by $f_1, \dots, f_n$ coordinate functions of $f$. By compactness arguments there is $\xi \in M$ such that $f_n(\xi) = \max_{p\in M}{f_n(p)}$. This means that $T_{0,\xi,\delta}(\mathbb{S}_{\xi}^{n-1})$ belongs to the hemisphere $x_n \leq 0$ and hence the modulo 2 degree of $T_{0,\xi,\delta}$ is zero. But there is a homotopy between $T_{0,\xi,\delta}$ and $T_{-\delta/2,\xi,\delta}$ by $s$. This is a contradiction. \qedhere
		
	\end{proof}
	
	\begin{definition}
		Let $n$ be a positive integer, and let $\delta$ be a positive real number. Let $M$ be a compact Riemannian manifold of dimension $n$, and let $f\colon M \to \mathbb{R}^n$ be a continuous map. We denote by $\mathcal{F}(\delta)$ the subset of $M \times M$ such that $\{a,b\} \in \mathcal{F}(\delta)$ if and only if $f(a) = f(b)$ and the points $a$ and $b$ are joined by a geodesic of length~$\delta$.
		
		We call points $a, b \in M$ $\delta$-\emph{conjugate} if they are joined by an infinite number of geodesics of length $\delta$. We denote the set of such points by $\mathcal{C}(\delta)$. 
	\end{definition}
	
	%Finally denote by $\mathcal{R}(\delta)$ a set of closed geodesics in $M$ of length $\delta$.

	\begin{theorem}
		Let $n$ be a positive integer such that $n > 1$, and let $\delta$ be a positive real number. Let $M$ be a compact Riemannian manifold of dimension $n$, and let $f\colon M \to \mathbb{R}^n$ be a continuous map. If $\mathcal{C}(\delta)$ is empty, then $\mathcal{F}(\delta)$ is uncountable. 
	\end{theorem}
	
	\begin{proof}
		
		If $f(M)$ is a singleton, then the result easily follows. Otherwise $\partial{f(M)}$ is uncountable (since $f(M)$ is bounded and connected). 
		
		If no $\xi$ in $\partial f(M)$ yields $(f^{-1}(\xi) \times f^{-1}(\xi)) \cap F(\delta) = \emptyset$, then the statement easily follows.
		
		Otherwise take $\xi \in \partial f(M)$ such that $(f^{-1}(\xi) \times f^{-1}(\xi)) \cap F(\delta) = \emptyset$ and let $q$ be a point in $f^{-1}(\xi)$. 
		
		Let $G_{\nu, q}$ be the radial projection with center at $\nu \in \mathbb{R}^n$ on a Euclidean sphere $\mathbb{S}^{n-1}$ centered at~$\nu$. 
		Denote by $\mathcal{S}_q^{n-1}$ the sphere of geodesic radius $\delta$ centered at~$q$, and let $\mathcal{O}$ be the connected component of $\mathbb{R}^n \setminus f(\mathcal{S}_q^{n-1})$ such that $f(q) \in \mathcal{O}$ (such a component exists since otherwise $(f^{-1}(\xi) \times f^{-1}(\xi)) \cap F(\delta) \neq \emptyset$). 
		
		There exists $\varepsilon > 0$ such that $U_{\varepsilon}(f(q)) \subset \mathcal{O}$, where $U_{\varepsilon}(f(q))$ is the $\varepsilon$-neighborhood of $f(q)$ in $\mathbb{R}^n$. Since $f(q)$ is in $\partial{f(M)}$, it follows that there exists a point $\mathcal{B} \in U_{\varepsilon}(f(q))$ such that $\mathcal{B} \notin f(M)$. Since $f(\mathcal{S}_q^{n-1})$ is contractible in $f(M)$, it is contractible in $\mathbb{R}^n \setminus \mathcal{B}$ as well. Hence the degree of $G_{\mathcal{B},q} \circ f\colon \mathcal{S}_q^{n-1} \to \mathbb{S}^{n-1}$ is zero as well as the degree of $G_{\xi,q} \circ f\colon \mathcal{S}_q^{n-1} \to \mathbb{S}^{n-1}$ and hence of ${T}_{0,q,\delta}$ (see the proof of the Hopf theorem for notation).

		From the continuity of~$f$ it follows that there is $\varepsilon_1 > 0$ such that for each $p \in U_{\varepsilon_1}(q)$, the degree of ${T}_{0,p,\delta}$ is zero, where $U_{\varepsilon_1}(q)$ is the $\varepsilon_1$-neighborhood of $q$ in $M$. First observe that no at most countable collection of geodesics in $M$ covers $U_{\varepsilon_1}(q)$, since any geodesic has zero $n$-dimensional Lebesgue measure. Second observe that through each point~$p$ in $U_{\varepsilon_1}(q)$ passes a geodesic $\gamma_p$ containing points in some pair in $\mathcal{F}(\delta)$ (the distance between points in this pair along $\gamma_p$ is $\delta$). Indeed, otherwise we could use the Hopf ‘trick’ and get a contradiction.

		Now if $\mathcal{F}(\delta)$ is at most countable, then there is a pair of points $\{\mu_1, \mu_2\} \in \mathcal{F}(\delta)$ such that $\mu_1$ and $\mu_2$ are joined by uncountably many geodesics of length $\delta$ and hence $\mathcal{C}(\delta)$ is not empty. This completes the proof. \qedhere

	\end{proof}

	%\bigskip
	

\begin{thebibliography}{9}
	
			\bibitem{E29}
		D. Attali, J-D. Boissonnat, H. Edelsbrunner, 
		{\it Stability and computation of medial axes: a state-of-the-art report}, 
		In: Mathematical Foundations of Scientific Visualization, Computer Graphics, and Massive Data Exploration, Springer, Berlin, 2009, pp.~109--125.		
		% MR2560510 (2012c:65034)
		%Stability and computation of medial axes: a state-of-the-art report. (English summary) Mathematical foundations of scientific visualization, computer graphics, and massive data exploration, 109вЂ“125, 
%Math. Vis., Springer, Berlin, 2009. 
		
		\bibitem{H44} 
		H. Hopf,
		{\it Eine Verallgemeinerung bekannter Abbildungs- und \"Uberdeckungss\"atze}, 
		Portugal. Math., {\bf 4} (1944), 129--139.
		
		\bibitem{L04}
		A. Lieutier, 
		{\it Any open bounded subset of $\mathrm{R}^n$ has the same homotopy type as its medial axis}, 
		Computer-Aided Design, {\bf 36} (2004), 1029--1046.
		%https://www.sciencedirect.com/science/article/abs/pii/S0010448504000065
		
		\bibitem{MM18} 
		A. V. Malyutin, O. R. Musin, 
		{\it Neighboring mapping points theorem}, 
		{\tt arXiv:1812.10895}, 2022, 1--21. 

		\bibitem{M88}	
		G. Matheron, 
		{\it Examples of topological properties of skeletons}, 
		In: Image Analysis and Mathematical Morphology. Vol.~2: Theoretical Advances,
		Academic Press, London, 1988, pp.~217--238.
				
		\bibitem{Th11} W. Thurston, A comment on \url{https://mathoverflow.net/questions/57766/why-are-there-no-wild-arcs-in-the-plane}. Accessed August 15, 2022.
		
		\bibitem{W47}		
		G. T. Whyburn, 
		{\it On locally simple curves}, 
		Bull. Amer. Math. Soc., {\bf 53} (1947), 986--992.
		
	\end{thebibliography}
\end{document}